\documentclass[12pt]{amsart} 
\usepackage{latexsym,enumerate,amscd,amssymb, 
hyperref}
\usepackage[T1]{fontenc}

\usepackage{fullpage}
\usepackage[nocolor]{sseq}
\usepackage[textwidth=1.2in, textsize=small]{todonotes}
\usepackage{tikz, tikz-cd}
\usetikzlibrary{arrows}
\usetikzlibrary{matrix}
\usetikzlibrary{shapes}
\usetikzlibrary{calc}
\usepgflibrary{shapes.geometric}
\usepackage{thmtools} 
\usepackage{thm-restate}
\usepackage{mathabx}

\theoremstyle{plain} 
\newtheorem{Theorem}{Theorem}[section]
\newtheorem{Corollary}[Theorem]{Corollary} 
\newtheorem{Lemma}[Theorem]{Lemma}

\theoremstyle{definition} 
\newtheorem{Definition}[Theorem]{Definition} 
\newtheorem{Question}[Theorem]{Question}
\newtheorem{Remark}[Theorem]{Remark} 
\newtheorem{example}[Theorem]{Example}
\newtheorem{obs}[Theorem]{Observation}

\numberwithin{equation}{section}
\numberwithin{figure}{section}

\newcommand{\G}{\Gamma}
\newcommand{\gG}{\Gamma}

\newcommand{\Conf}[2]{\mathrm{Conf}^\square_{#1}({#2})} 

\newcommand{\cP}{\mathcal P}

\DeclareMathOperator{\Lk}{Lk}

\DeclareMathOperator{\cConf}{Conf}

\title{The algebraic structure of hyperbolic graph braid groups}
\author{B.~Appiah, P.~Dani, W.~Ge, C.~Hudson, S.~Jain, M.~Lemoine, J.~Murphy, J.~Murray, A.~Pandikkadan, K.~Schreve, and H.~Vo}
\thanks{This material is based upon work supported by the National Science Foundation under Award No.~2231492 and Award No.~2203325}

\begin{document}\begin{abstract}
Genevois  recently classified which graph braid groups on $\ge 3$ strands are word hyperbolic. 
In the $3$-strand case, he asked whether all such word hyperbolic groups are actually free;  this reduced to checking two infinite classes of graphs: \emph{sun} and \emph{pulsar} graphs. We prove that $3$-strand braid groups of sun graphs are free.  On the other hand, it was known to experts that $3$-strand braid groups of most pulsar graphs contain surface subgroups.  We provide a simple proof of this and prove an additional structure theorem for these groups.

\end{abstract}

\maketitle
\section{Introduction}
 
Let $\gG$ be a finite graph. The main object of study in this paper is the configuration space of $n$ unordered points on $\gG$; $$\cConf_n(\gG) := (\gG^n - \Delta)/S_n$$ where $\Delta$ is the diagonal in $\gG^n$ and the $S_n$-action permutes the factors. The fundamental group of $\cConf_n(\gG)$ is the so-called \emph{$n$-strand graph braid group}, which we will denote by $B_n(\gG)$. 

The algebraic structure of graph braid groups has been studied from various viewpoints, see for example~\cite{abrams-ghrist, farley-sabalka2, Ko-Park}. 
Recently, Genevois~\cite{Genevois}, and later Berlyne~\cite{berlyne}  studied these groups from the perspective of geometric group theory, and we continue in this vein. 
In fact, these groups all come equipped with natural well-behaved geometric structures. 
Abrams and \'{S}wi\k{a}tkowski
 independently showed that $\cConf_n(\gG)$ deformation retracts onto a locally CAT(0) cube complex~\cite{Abrams, Swiat}. Crisp and Wiest later showed that this cube complex admits a local isometry to the Salvetti complex of a right-angled Artin group, i.e. in modern language, graph braid groups are compact special~\cite{crisp-wiest}. 

\subsection{A discretized version}

A discretized analogue of $\cConf_n(\gG)$ is not hard to describe; one simply removes a larger neighborhood of the diagonal. To be precise, if $(e_1, \dots e_n)$ denotes an $n$-tuple of cells in $\gG$, define $\Delta^\square$ as  $$ \{ (e_1, e_2, \dots e_n) \;|\; e_i \cap e_j \ne \varnothing \text{ for some } i \ne j\},$$ and let $$\Conf{n}{\gG} = (\gG^n - \Delta^\square)/S_n$$ 

Then $\Conf{n}{\gG}$ is naturally a cube complex; $k$-cubes in $\Conf{n}{\gG}$ can be identified with a choice of $k$ disjoint edges in $\gG$ along with $n-k$ disjoint vertices. 
If one isn't careful, $\Conf{n}{\gG}$ has no relation to $\cConf_n(\gG)$; for example if $n$ is greater than the number of vertices of $\gG$ then $\Conf{n}{\gG}$ is empty. 
On the other hand, Abrams showed that a suitable subdivision of $\gG$ always leads to a $\Conf{n}{\gG}$ that is a deformation retract of $\cConf_n(\gG)$.   It will be convenient  for us to subdivide $\Gamma$ as little as possible, so we will use a theorem of Prue and Scrimshaw~\cite[Thereom 3.2]{prue-scrimshaw} which says that if $\Gamma$ is a graph with at least $n$ vertices, then $\cConf_n(\gG)$ deformation retracts to $\Conf{n}{\gG}$ provided each path between distinct vertices of degree $\ne 2$ has length at least $n-1$ and each homotopically essential loop has length at least $n+1$.   We call this an \emph{admissible} subdivision.

It follows quickly from Gromov's link condition that $\Conf{n}{\gG}$ is locally CAT(0). Indeed, 
a $(k+1)$-clique (i.e.~complete subgraph with $k+1$ vertices) 
in the link of a vertex $\{v_1, \dots, v_n\}$ of $\Conf{n}{\gG}$ corresponds to $(k+1)$-pairwise disjoint edges in $\gG$, where each edge  is incident to one of the $v_i$. Since the edges are pairwise disjoint, they correspond to a $(k+1)$-cube in $\Conf{n}{\gG}$ containing $\{v_1, \dots, v_n\}$, and this cube fills the $(k+1)$-clique with a $k$-simplex, hence all links are flag complexes. 

\subsection{Hyperbolic graph braid groups}

Genevois analyzed when graph braid groups are word hyperbolic groups in the sense of Gromov. This turns out to be a very rare occurence. In particular he showed:

\begin{Theorem}\cite[Theorem 1.1]{Genevois}
Let $\gG$ be a finite graph. Then 

\begin{itemize}
\item $B_2(\gG)$ is hyperbolic if and only if $\gG$ does not contain a pair of disjoint cycles. 
\item $B_3(\gG)$ is hyperbolic if and only if $\gG$ is a tree, a rose graph, a sun graph, or a pulsar graph. 
\item For $n \ge 4$, $B_n(\gG)$ is hyperbolic if and only if $\gG$ is a rose graph. 
\end{itemize}
\end{Theorem}

Here, a \emph{rose graph} is a wedge of circles and rays, a \emph{sun graph} is obtained from a cycle by attaching rays to some vertices (see Figure~\ref{SunGraph} and Definition~\ref{def:sun}), and a \emph{pulsar graph} is obtained by gluing cycles along a fixed segment  in each, and attaching rays to the endpoints of this segment (see Figure~\ref{fig:theta_m} and Definition~\ref{def:pulsar}).  
A special case of a pulsar graph is the \emph{generalized theta graph} $\Theta_m$, which is the suspension of $m$ points. 

In each case, the non-hyperbolicity of $B_n(\gG)$ comes from obvious free abelian subgroups. If $\gG$ has two disjoint cycles $C_1$ and $C_2$, then $B_2(\gG)$ contains the subgroup $\mathbb{Z}^2 \cong \pi_1(C_1 \times C_2)$.
If $\gG$ has a cycle that is disjoint from a vertex $v$ of degree $>2$, then $B_3(\gG)$ contains a $\mathbb{Z}^2$-subgroup corresponding to one particle moving around the cycle and two particles swapping places around $v$. Finally, if $\gG$ contains two vertices of degree $>2$, then the $\mathbb{Z}^2$-subgroup of $B_4(\gG)$ comes from swapping disjoint pairs of particles around the two vertices. 

Genevois showed that if $\gG$ is a rose graph, then $B_n(\gG)$ is free for all $n$.  It is known that $B_3(\gG)$ is free when $\gG$ is a tree~\cite{farley-sabalka}.  
Motivated by this, Genevois asked about the the other cases when $B_3(\gG)$ is hyperbolic: 

\begin{Question}\label{q:Genevois} \cite[Question 5.3]{Genevois}
Are the $3$-strand graph braid groups of sun graphs and pulsar graphs free?   one-ended? surface groups? 3-manifold groups? 
\end{Question}

Our first theorem confirms that the 3-strand  braid groups of sun graphs are indeed free:

\begin{Theorem} \label{t:sun}
If $\mathcal S$ is a sun graph, then $B_3(\mathcal S)$ is free.
\end{Theorem}

On the other hand, it turns out that the $3$-strand  graph braid group of a pulsar graph is almost never free. 
In fact it was already known that $\cConf_3(\Theta_4)$ is homotopy equivalent to a closed, orientable surface of genus $3$~\cite{SL, W-G}. 
We provide a simple proof of this: we show that with the simplest admissible subdivision of $\Theta_4$, the space $\Conf{3}{\Theta_4}$ is \emph{homeomorphic} to this surface. 
Since   $B_3(\Theta_4)$ injects into $B_3(\cP_{m, n_1, n_2})$ by~\cite[Corollary 3.14]{Abrams}, it follows that $B_3(\cP_{m, n_1, n_2})$ is not free for $m \ge 4$. 

In this case, we still prove some additional structure for the 3-strand graph braid groups of pulsar graphs: 
\begin{Theorem} \label{t:split}
If $\cP$ is a pulsar graph constructed by attaching rays to the generalized theta graph $\Theta_m$, then $B_3(\cP) = B_3(\Theta_m) \ast \mathbb{F}$, where $\mathbb{F}$ is a non-trivial free group of finite rank.  
\end{Theorem}

This still has implications for Question~\ref{q:Genevois}, as it shows that $B_3(\cP)$ is not one-ended when $\cP$ has a non-trivial ray, i.e.~when it is not a generalized theta graph.

Next, we compute the Euler characteristics of the configurations spaces  
$\cConf_3(\Theta_m)$:

\begin{Theorem}
$\chi(\cConf_3(\Theta_m))=
\frac{m(m-2)(m-7)}{6}.$
\end{Theorem}
In particular, $\chi(\cConf_3(\Theta_7)) =0$ and $\chi(\cConf_3(\Theta_m)) >0$ when $m>7$. For $m>7$, this implies that $B_3(\Theta_m)$ is not the fundamental group of an aspherical $3$-manifold (a doubling argument shows these have nonpositive Euler characteristic).  This answers another part of Question~\ref{q:Genevois}.

Our main tool to prove Theorems \ref{t:sun} and \ref{t:split} is the PL Morse theory of Bestvina--Brady~ \cite{bestvina-brady}.  This was inspired by Bestvina's proof that braid groups of the tripod graph are free~\cite[Section 4]{bestvina}.
Discrete notions of Morse theory have been used to study graph braid groups before, however most of the earlier papers use Forman's version~\cite{forman}.  

We provide background on PL Morse theory in Section~\ref{sec:PL}, and our results on sun graphs and pulsar graphs are proved in Sections~\ref{sec:sun} and~\ref{sec:pulsar} respectively.

This paper arose from a VIR (Vertically Integrated Research) course on the topic in the Department of 
Mathematics at Louisiana State University.  

We thank Ben Knudsen for comments on an earlier version of this paper.

\section{PL Morse Theory}\label{sec:PL}

We begin this section by reviewing the necessary background on PL Morse theory in the context of cube complexes.  Then in Section~\ref{sec:conf Morse} we describe a natural way to define a Morse function on a discretized configuration space, and provide examples illustrating the computations of the associated descending links.

\subsection{PL Morse theory for cube complexes}

Let $X$ be a finite cube complex. For each cube $c \in X$, consider a characteristic function $\chi_c: C_c \to X$ where $C_c$ is $[0,1]^n$. A function $f: X \to \mathbb{R}$ is said to be a PL Morse function if the following are satisfied.

\begin{enumerate}
    \item For every cube $c$, $f\chi_c: C_c \to \mathbb{R}$ is the restriction of an affine map $\mathbb{R}^m \to \mathbb{R}$.
    \item If the image under $f$ of a cube $c$ is constant, then $c$ is 0-dimensional.
\end{enumerate}

Note that distinct $0$-dimensional cells can have the same image under $f$.  For example, in 
Figure~\ref{fig:level set}, there are three 0-cells in $f^{-1} (b)$. 
Let $v$ be a vertex of $X$. A cube $c$ with $v \in c$ is \emph{descending} at $v$ if $f(v)$ is maximal in the image $f(c)$. The \emph{descending link} at $v$, denoted by $\Lk_{\downarrow}(v)$, is the link of $v$ in the union of the descending cells at $v$.
For example, in Figure~\ref{fig:level set}, $\Lk_{\downarrow}(v_0)$ consists of the two blue vertices shown and the blue edge connecting them, while $\Lk_{\downarrow}(v_4)$ is empty.

The statements and proofs of the following two lemmas above are presented in~\cite[Proposition 2.4 and 2.6]{bestvina}.

\begin{Lemma}
\label{lemma_novertex}
 Let $J = [a,b] \subset \mathbb{R}$ be a closed interval such that $(a,b]$ is
disjoint from $f(X^0)$. Then $f^{-1}(J)$ deformation retracts to $f^{-1}(a)$
\end{Lemma}

\begin{Lemma}
\label{lemma_vertex}
Suppose that $J = [a,b]$ is a closed interval and that $f^{-1}(J)$ contains one vertex $v$ and $f(v) = b$. Then the pair $(f^{-1}(J),f^{-1}(a))$ is homotopy equivalent rel $f^{-1}(a)$ to $(Q,f^{-1}(a))$ where $Q$ is $f^{-1}(a)$ with the cone on $\Lk_{\downarrow}(v)$ attached.

\end{Lemma}

    Note that if $f^{-1}(b)$ contains more than one vertex, then $Q$ is $f^{-1}(a)$ with the cones on the descending links of all the vertices attached, see Figure~\ref{fig:level set}.

\begin{figure}%[htp]
\begin{tikzpicture}[scale = .75]
    \draw[fill=lightgray] (0,0)--(3,3)--(6,0)--(3,-3)--(0,0);
    \draw[fill=red] (0,0) circle (3pt) node[above=0.1] { $v_1$};
    \draw[fill=black] (1.5,1.5) circle (3pt);
    \draw[fill=black] (3,3) circle (3pt) node[above=0.1] { $v_0$};
    \draw[fill=black] (4.5,1.5) circle (3pt);
    \draw[fill=red] (6,0) circle (3pt) node[above=0.1] { $v_3$};
    \draw[fill=black] (4.5,-1.5) circle (3pt);
    \draw[fill=black] (3,-3) circle (3pt)  node[below=0.1] { $v_4$};
    \draw[fill=black] (1.5,-1.5) circle (3pt);
    \draw[thick] (0,0)--(1.5,1.5)--(3,3)--(4.5,1.5)--(6,0)--(4.5,-1.5)--(3,-3)--(1.5,-1.5)--(0,0);
    \draw[thick] (1.5,1.5)--((4.5,-1.5);
    \draw[thick] (4.5,1.5)--(1.5,-1.5);
    \draw[thick,->] (9,-4)--(9,4) node[right] {$\mathbb{R}$};
    \draw[gray,dashed] (0,0)--(9,0) node[right=0.1] { $b$};
    \draw[gray,dashed] (1,-1)--(9,-1) node[right=0.1] { $a$};
    \draw[red,thick] (0,0)--(1,-1)--(5,-1)--(6,0);
    \draw[fill=red] (2,-1)--(3,0)--(4,-1)--(2,-1);
    \draw[fill=red] (3,0) circle (3pt) node[above=0.1] { $v_2$};
        \draw[blue,thick](2.3, 2.3)--(3.7, 2.3);
 \draw[fill=blue] (2.3, 2.3) circle (2pt);
 \draw[fill=blue] (3.7, 2.3) circle (2pt);
\end{tikzpicture}
\caption{
Let $X \subset \mathbb{R}^n$ be the above finite cube complex, and define a PL Morse function $f:X\to \mathbb R$ by projection to the $y$-axis.  
Then $\Lk_{\downarrow}(v_0)$ is shown in blue, while $\Lk_{\downarrow}(v_4)$ is empty. 
Let  $J = [a, b]$.  Then $f^{-1}(b)\cap X^{0}$ consists of the three vertices $v_1$, $v_2$, and $v_3$,  and $f^{-1}(J)$ deformation retracts to the red space -- the union of the horizontal segment $f^{-1}(a)$ and the cones 
on the descending links of the three vertices. }
\label{fig:level set}
\end{figure}
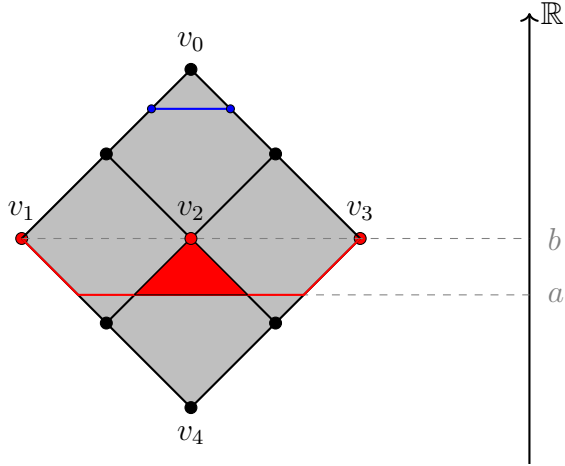

\begin{Corollary}\label{c:freeproduct}
Let the image of the Morse function $f$ be an interval $[a,c]\subset \mathbb{R}$ and let $b \in (a, c)$. If
\begin{enumerate}[(a)]
    \item $f^{-1}([a,b])$ is connected
    \item $f^{-1}((b,c))$ has no vertices
    \item $f^{-1}(c)$ is a set of vertices each having descending links a union of contractible spaces
\end{enumerate}
then $f^{-1}([a,c])$ is homotopy equivalent to a wedge of $f^{-1}([a,b])$ with a collection of circles and consequently
$$\pi _1(f^{-1}([a,c]))=\pi _1(f^{-1}([a,b]))\ast \mathbb{F}.$$
where $\mathbb{F}$ is a free group on a finite number of generators.    
\end{Corollary}

\begin{proof}
Let us assume that $f^{-1}(c)$ is a single vertex $v$ and denote $f^{-1}([a,b])$ by $Z$. Given that  $Z$ is connected and $\Lk_{\downarrow}(v)$ is homotopy equivalent to a finite set of points, $f^{-1}([a,c])$ is homotopy equivalent to $Z\cup \displaystyle{\bigvee_{\text{finite}}} S^1$ (see Figure \ref{f:wedge}) by Lemmas~\ref{lemma_novertex} and \ref{lemma_vertex}. The same argument works when $f^{-1}(c)$ is a finite set of vertices. 
\end{proof}

\tikzset{every picture/.style={line width=0.75pt}} %set default line width to 0.75pt        
\begin{figure}%[htp]
\centering
\begin{tikzpicture}[x=0.75pt,y=0.75pt,yscale=-1,xscale=1]
%Shape: Polygon Curved [id:ds08997707145429412] 
\draw  [fill=lightgray  ,fill opacity=1 ][line width=1.5]  (136,168) .. controls (156,158) and (187,189.88) .. (237,169.88) .. controls (287,149.88) and (295,185.88) .. (280,214.88) .. controls (265,243.88) and (172,250.88) .. (136,228) .. controls (100,205.12) and (116,178) .. (136,168) -- cycle ;
%Straight Lines [id:da8030714859050365] 
\draw [color={rgb, 255:red, 208; green, 2; blue, 27 }  ,draw opacity=1 ][fill={rgb, 255:red, 231; green, 25; blue, 25 }  ,fill opacity=1 ][line width=1.5]    (182,93.88) -- (214,187.88) ;
\draw [shift={(214,187.88)}, rotate = 71.2] [color={rgb, 255:red, 208; green, 2; blue, 27 }  ,draw opacity=1 ][fill={rgb, 255:red, 208; green, 2; blue, 27 }  ,fill opacity=1 ][line width=1.5]      (0, 0) circle [x radius= 4.36, y radius= 4.36]   ;
\draw [shift={(182,93.88)}, rotate = 71.2] [color={rgb, 255:red, 208; green, 2; blue, 27 }  ,draw opacity=1 ][fill={rgb, 255:red, 208; green, 2; blue, 27 }  ,fill opacity=1 ][line width=1.5]      (0, 0) circle [x radius= 4.36, y radius= 4.36]   ;
%Straight Lines [id:da3235784492884448] 
\draw [color={rgb, 255:red, 208; green, 2; blue, 27 }  ,draw opacity=1 ][line width=1.5]    (182,93.88) -- (178,189.88) ;
\draw [shift={(178,189.88)}, rotate = 92.39] [color={rgb, 255:red, 208; green, 2; blue, 27 }  ,draw opacity=1 ][fill={rgb, 255:red, 208; green, 2; blue, 27 }  ,fill opacity=1 ][line width=1.5]      (0, 0) circle [x radius= 4.36, y radius= 4.36]   ;
\draw [shift={(182,93.88)}, rotate = 92.39] [color={rgb, 255:red, 208; green, 2; blue, 27 }  ,draw opacity=1 ][fill={rgb, 255:red, 208; green, 2; blue, 27 }  ,fill opacity=1 ][line width=1.5]      (0, 0) circle [x radius= 4.36, y radius= 4.36]   ;
%Straight Lines [id:da408337859561522] 
\draw [color={rgb, 255:red, 208; green, 2; blue, 27 }  ,draw opacity=1 ][fill={rgb, 255:red, 208; green, 2; blue, 27 }  ,fill opacity=1 ][line width=1.5]    (182,93.88) -- (197,191.88) ;
\draw [shift={(197,191.88)}, rotate = 81.3] [color={rgb, 255:red, 208; green, 2; blue, 27 }  ,draw opacity=1 ][fill={rgb, 255:red, 208; green, 2; blue, 27 }  ,fill opacity=1 ][line width=1.5]      (0, 0) circle [x radius= 4.36, y radius= 4.36]   ;
\draw [shift={(182,93.88)}, rotate = 81.3] [color={rgb, 255:red, 208; green, 2; blue, 27 }  ,draw opacity=1 ][fill={rgb, 255:red, 208; green, 2; blue, 27 }  ,fill opacity=1 ][line width=1.5]      (0, 0) circle [x radius= 4.36, y radius= 4.36]   ;
%Shape: Polygon Curved [id:ds30914348348016585] 
\draw  [fill=lightgray  ,fill opacity=1 ][line width=1.5]  (439,169) .. controls (459,159) and (490,190.88) .. (540,170.88) .. controls (590,150.88) and (598,186.88) .. (583,215.88) .. controls (568,244.88) and (475,251.88) .. (439,229) .. controls (403,206.12) and (419,179) .. (439,169) -- cycle ;
%Curve Lines [id:da7890254220666574] 
\draw [color={rgb, 255:red, 208; green, 2; blue, 27 }  ,draw opacity=1 ][line width=1.5]     ;
%Shape: Triangle [id:dp35398259099057583] 
\draw  [color={rgb, 255:red, 208; green, 2; blue, 27 }  ,draw opacity=1 ][fill={rgb, 255:red, 233; green, 107; blue, 119 }  ,fill opacity=1 ][line width=1.5]  (182,93.88) -- (162.21,185.88) -- (139,185.88) -- cycle ;
%Straight Lines [id:da8767964195142647] 
\draw [color={rgb, 255:red, 208; green, 2; blue, 27 }  ,draw opacity=1 ][line width=1.5]    (182,93.88) -- (162.21,185.88) ;
\draw [shift={(162.21,185.88)}, rotate = 102.14] [color={rgb, 255:red, 208; green, 2; blue, 27 }  ,draw opacity=1 ][fill={rgb, 255:red, 208; green, 2; blue, 27 }  ,fill opacity=1 ][line width=1.5]      (0, 0) circle [x radius= 4.36, y radius= 4.36]   ;
\draw [shift={(182,93.88)}, rotate = 102.14] [color={rgb, 255:red, 208; green, 2; blue, 27 }  ,draw opacity=1 ][fill={rgb, 255:red, 208; green, 2; blue, 27 }  ,fill opacity=1 ][line width=1.5]      (0, 0) circle [x radius= 4.36, y radius= 4.36]   ;
%Straight Lines [id:da14683037389074316] 
\draw [color={rgb, 255:red, 208; green, 2; blue, 27 }  ,draw opacity=1 ][line width=1.5]    (182,93.88) -- (139,185.88) ;
\draw [shift={(139,185.88)}, rotate = 115.05] [color={rgb, 255:red, 208; green, 2; blue, 27 }  ,draw opacity=1 ][fill={rgb, 255:red, 208; green, 2; blue, 27 }  ,fill opacity=1 ][line width=1.5]      (0, 0) circle [x radius= 4.36, y radius= 4.36]   ;
\draw [shift={(182,93.88)}, rotate = 115.05] [color={rgb, 255:red, 208; green, 2; blue, 27 }  ,draw opacity=1 ][fill={rgb, 255:red, 208; green, 2; blue, 27 }  ,fill opacity=1 ][line width=1.5]      (0, 0) circle [x radius= 4.36, y radius= 4.36]   ;
%Straight Lines [id:da2611643665158182] 
\draw [color={rgb, 255:red, 208; green, 2; blue, 27 }  ,draw opacity=1 ]   (139,185.88) -- (162.21,185.88) ;
\draw [shift={(162.21,185.88)}, rotate = 0] [color={rgb, 255:red, 208; green, 2; blue, 27 }  ,draw opacity=1 ][fill={rgb, 255:red, 208; green, 2; blue, 27 }  ,fill opacity=1 ][line width=0.75]      (0, 0) circle [x radius= 3.35, y radius= 3.35]   ;
\draw [shift={(139,185.88)}, rotate = 0] [color={rgb, 255:red, 208; green, 2; blue, 27 }  ,draw opacity=1 ][fill={rgb, 255:red, 208; green, 2; blue, 27 }  ,fill opacity=1 ][line width=0.75]      (0, 0) circle [x radius= 3.35, y radius= 3.35]   ;
%Curve Lines [id:da49436835593441397] 
\draw [color={rgb, 255:red, 208; green, 2; blue, 27 }  ,draw opacity=1 ][line width=1.5]    (467.5,189) .. controls (416.5,169) and (370.5,111) .. (411.5,108) .. controls (452.5,105) and (438.5,117) .. (467.5,189) -- cycle ;
\draw [shift={(467.5,189)}, rotate = 201.41] [color={rgb, 255:red, 208; green, 2; blue, 27 }  ,draw opacity=1 ][fill={rgb, 255:red, 208; green, 2; blue, 27 }  ,fill opacity=1 ][line width=1.5]      (0, 0) circle [x radius= 4.36, y radius= 4.36]   ;
%Curve Lines [id:da48311244087994965] 
\draw [color={rgb, 255:red, 208; green, 2; blue, 27 }  ,draw opacity=1 ][line width=1.5]    (467.5,189) .. controls (432.5,134) and (416.5,96) .. (455.5,92) .. controls (494.5,88) and (483.5,117) .. (467.5,189) -- cycle ;
\draw [shift={(467.5,189)}, rotate = 237.53] [color={rgb, 255:red, 208; green, 2; blue, 27 }  ,draw opacity=1 ][fill={rgb, 255:red, 208; green, 2; blue, 27 }  ,fill opacity=1 ][line width=1.5]      (0, 0) circle [x radius= 4.36, y radius= 4.36]   ;
%Curve Lines [id:da9576312282708673] 
\draw [color={rgb, 255:red, 208; green, 2; blue, 27 }  ,draw opacity=1 ][line width=1.5]    (467.5,189) .. controls (504.5,124) and (471.5,99) .. (510.5,95) .. controls (549.5,91) and (531.5,160) .. (467.5,189) -- cycle ;
\draw [shift={(467.5,189)}, rotate = 299.65] [color={rgb, 255:red, 208; green, 2; blue, 27 }  ,draw opacity=1 ][fill={rgb, 255:red, 208; green, 2; blue, 27 }  ,fill opacity=1 ][line width=1.5]      (0, 0) circle [x radius= 4.36, y radius= 4.36]   ;

% Text Node
\draw (238,190) node [anchor=north west][inner sep=0.75pt]   [align=left] {Z};
% Text Node
\draw (538,193) node [anchor=north west][inner sep=0.75pt]   [align=left] {Z};
% Text Node
\draw (157,85) node [anchor=north west][inner sep=0.75pt]  [color={rgb, 255:red, 208; green, 2; blue, 27 }  ,opacity=1 ] [align=left] {\textcolor[rgb]{0.82,0.01,0.11}{v}};
% Text Node
\draw (337,172.4) node [anchor=north west][inner sep=0.75pt]    {$\simeq $};
\end{tikzpicture}
\caption{In the figure on the left, $Z \cup  \text{Cone($v$,$\Lk_{\downarrow}(v)$)}$ is the deformation retract of $f^{-1}([a,c])$ and is homotopy equivalent to $Z\cup \displaystyle{\bigvee_{\text{finite}}} S^1$ (on the right). }
\label{f:wedge}
\end{figure}
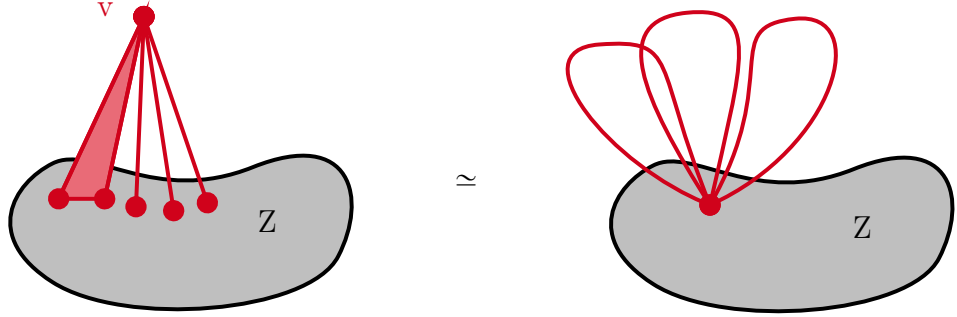

\begin{Remark}
If the descending links of all the vertices in a connected cube complex are unions of contractible spaces then it is homotopy equivalent to a wedge of circles.    
\end{Remark}

\subsection{Morse functions for configuration spaces} \label{sec:conf Morse}

We now describe the PL Morse functions we will use on the configuration spaces  $\Conf{n}{\gG}$.  In each of our examples, we choose an embedding of $\G$ 
into $\mathbb{R}^2$ with no horizontal edges, and we define 
a PL Morse function 
$h: \G \to \mathbb{R}$ as the restriction of the projection of $\mathbb{R}^2$ onto the $y$-axis.  
This induces a PL Morse function $f: \Conf{n}{\gG} \rightarrow \mathbb{R}$
as follows: for $\{y_1,y_2,\dots, y_n\} \in \Conf{n}{\gG}$, we define 
$$f(\{y_1,y_2,\dots, y_n\}) = h(y_1) + \dots + h(y_n).$$

Note that as $\G$ is finite, the Morse function $f$ attains a minimum.  
For any vertex $\{v_1, \dots, v_n\}$ 
of $\Conf{n}{\gG} $ we can find its descending link by looking at the possible ways in which each of the $v_i$ 
 can  ``move down'' in $\G$.  This is illustrated in the following example.

%\begin{center}
\begin{figure}%[hp]
\begin{tikzpicture}[scale = .75]
%%% edges
\draw[thick] (1,0) -- (2.732,1) -- (2.732,3) -- (1,4) -- (-0.732,3) --(-0.732,1) -- (1,0);
%% vertices
\draw[fill=black] (1,0) circle (3pt);
\draw[fill=black] (1,4) circle (3pt);
\draw[fill=black] (2.732,1) circle (3pt);
\draw[fill=black] (2.732,3) circle (3pt);
\draw[fill=black] (-0.732,1) circle (3pt);
\draw[fill=black] (-0.732,3) circle (3pt);
%% vertex labels
\node at (1,-0.5) {$a_1$};
\node at (3.232,1) {$a_2$};
\node at (3.232,3) {$a_3$};
\node at (1,4.5) {$a_4$};
\node at (-1.232,3) {$a_5$};
\node at (-1.232,1) {$a_6$};
\draw[thick,->] (7,-1.5)--(7,5.5) node[right] {$\mathbb{R}$};
\draw[gray,dashed] (1,4)--(7,4) node[right=0.2] {\large $b$};
\draw[gray,dashed] (1,0)--(7,0) node[right=0.2] {\large $a$};
\end{tikzpicture}
\caption{The height function for the cycle graph $\Gamma$.}\label{f:cycle}

\label{fig:c6_heightfn}
\end{figure}
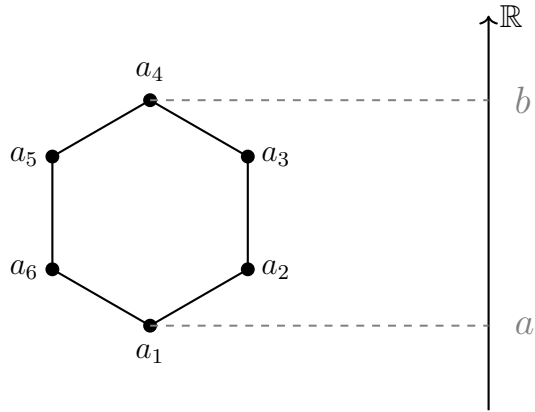
%\end{center}
%\begin{center}
\begin{figure}%[hp]
\begin{tikzpicture}[scale = .75]
%%% edges
\draw[thick] (0,0) -- (1.732,1) -- (1.732,3) -- (0,4) -- (-1.732,3) --(-1.732,1) -- (0,0);
%% vertices
\draw[fill=cyan] (0,0) circle (3pt);
\draw[fill=black] (0,4) circle (3pt);
\draw[fill=cyan] (1.732,1) circle (3pt);
\draw[fill=black] (1.732,3) circle (3pt);
\draw[fill=black] (-1.732,1) circle (3pt);
\draw[fill=cyan] (-1.732,3) circle (3pt);
%% vertex labels
\node at (0,-0.4) {$a_1$};
\node at (2.232,1) {$a_2$};
\node at (2.232,3) {$a_3$};
\node at (0,4.4) {$a_4$};
\node at (-2.232,3) {$a_5$};
\node at (-2.232,1) {$a_6$};

\begin{scope}[xshift=6.5cm]
   %%% edges
\draw[thick] (0,0) -- (1.732,1) -- (1.732,3) -- (0,4) -- (-1.732,3) --(-1.732,1) -- (0,0);
%% vertices
\draw[fill=magenta] (0,0) circle (3pt);
\draw[fill=magenta] (0,4) circle (3pt);
\draw[fill=black] (1.732,1) circle (3pt);
\draw[fill=magenta] (1.732,3) circle (3pt);
\draw[fill=black] (-1.732,1) circle (3pt);
\draw[fill=black] (-1.732,3) circle (3pt);
%% vertex labels
\node at (0,-0.4) {$a_1$};
\node at (2.232,1) {$a_2$};
\node at (2.232,3) {$a_3$};
\node at (0,4.4) {$a_4$};
\node at (-2.232,3) {$a_5$};
\node at (-2.232,1) {$a_6$};

\end{scope}
\begin{scope}[xshift = 13cm]
\draw[thick] (0,0) -- (1.732,1) -- (1.732,3) -- (0,4) -- (-1.732,3) --(-1.732,1) -- (0,0);
%% vertices
\draw[fill=black] (0,0) circle (3pt);
\draw[fill=blue] (0,4) circle (3pt);
\draw[fill=blue] (1.732,1) circle (3pt);
\draw[fill=black] (1.732,3) circle (3pt);
\draw[fill=blue] (-1.732,1) circle (3pt);
\draw[fill=black] (-1.732,3) circle (3pt);
%% vertex labels
\node at (0,-0.4) {$a_1$};
\node at (2.232,1) {$a_2$};
\node at (2.232,3) {$a_3$};
\node at (0,4.4) {$a_4$};
\node at (-2.232,3) {$a_5$};
\node at (-2.232,1) {$a_6$};
\end{scope}

\begin{scope}[xshift = 7cm, yshift = -6cm]
\draw[thick, fill=lightgray] (0,5) -- (-2,3) -- (-2,0) -- (0,2) --(0,5);
\draw[thick,red] (-1,4) -- (0,3.5);
\draw[fill=magenta] (0,5) circle (3pt)node[right=0.05, magenta] { $a_1a_3a_4$};
\draw[fill=black] (-2,3) circle (3pt) node[left=0.05] { $a_1a_3a_5$};
\draw[fill=black] (-2,0) circle (3pt) node[below=0.05] { $a_1a_2a_5$};
\draw[fill=black] (0,2) circle (3pt) node[right=0.05] { $a_1a_2a_4$};
\end{scope}
\begin{scope}[ yshift = -4cm]
\node  at (0,2) [label={[xshift=0cm, yshift=-0.1cm, cyan]$a_1a_2a_5$}] { };
%\node at (0,0.5) [label={[xshift=0cm, yshift=-0.6cm]$a_1a_2a_6$}] { };
%%% edges
\draw[thick, fill=lightgray] (0,2) -- (0,0.5);
\draw[fill=red] (0,1.5) circle (2pt);
\draw[fill=cyan] (0,2) circle (3pt);
\draw[fill=black] (0,0.5) circle (3pt) node[below=0.1] { $a_1a_2a_6$};
\end{scope}

\begin{scope}[xshift = 13cm, yshift = -4cm]
\fill[black] (0,2) circle (3pt)node[above=0.1] { $a_2a_5a_6$};
\fill[black] (2,0) circle (3pt) node[right=0.1] { $a_1a_4a_6$};
\fill[black] (2,2) circle (3pt) node[above=0.1] { $a_1a_5a_6$};
\fill[black] (0,-2) circle (3pt)node[below=0.1] { $a_2a_3a_6$};
\fill[black] (-2,0) circle (3pt)node[left=0.05] { $a_1a_2a_4$};
\fill[black] (-2,-2) circle (3pt)node[below=0.1] { $a_1a_2a_3$};
\fill[black] (2,-2) circle (3pt)node[below=0.1] { $a_1a_3a_6$};
\fill[black] (-2,2) circle (3pt)node[above=0.1] { $a_1a_2a_5$};
\draw[thick] (-2,2) -- (2, 2);
\draw[thick] (-2,2) -- (-2, -2);
\draw[thick] (0,2) -- (0, -2);
\draw[thick] (2,0) -- (-2, 0);
\draw[thick] (2,2) -- (2, -2);
\draw[thick] (-2,-2) -- (2, -2);
\draw[thick, red] (-0.6,0) -- (0, 0.6);
\draw[thick, red] (0.6,0) -- (0, 0.6);
\draw[thick, red] (0.6, 0) -- (0, -0.6);
\draw[thick, red] (-0.6,0) -- (0, -0.6);
\fill[blue] (0,0) circle (3pt)node[above right=0.13] { $a_2a_4a_6$};
\end{scope}

\end{tikzpicture}
\caption{The vertices $\{a_1,a_2,a_5\}, \{a_1,a_3,a_4\}$ and $\{a_2, a_4, a_6\}$ in $\Conf{3}{\gG}$ and their descending links.}
\label{fig:c6_egv1}
\end{figure}
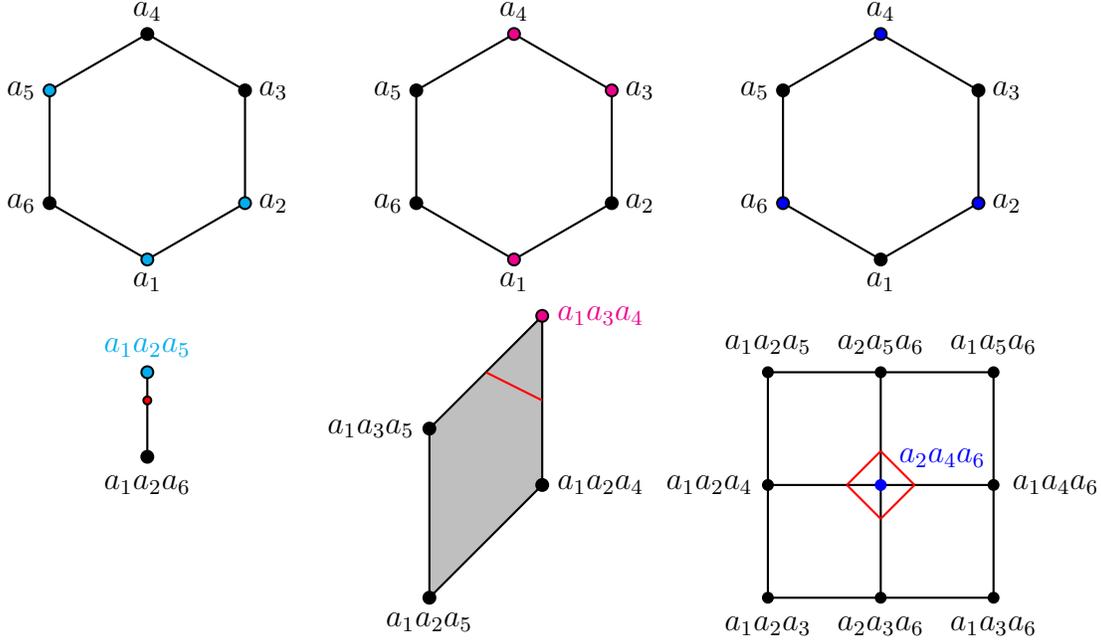

\begin{example}
Consider a cycle graph $\gG$ with six vertices. Figure \ref{f:cycle} shows an 
embedding of $\Gamma$ into the plane, and this induces a Morse function $f$ on $\Conf{3}{\gG}$. 
Notice that the local minima of $f$ are attained at the vertices  $\{a_1, a_2, a_6\}$, $\{a_1, a_5, a_6\}$, and $\{a_1, a_2, a_3\}$,  because for each $a_i$ in each of these triples, there is no way to move down in $\G$ to another unoccupied vertex.
Hence there are no descending cells at these vertices, and their descending links are empty. 
We will discuss 
the descending links of 
a few other vertices in $\Conf{3}{\gG}$ so that the reader can gain some intuition, see Figure~\ref{fig:c6_egv1}.

The descending link of the vertex $\{a_1, a_2, a_5\}$ is a single point (red in the figure), since there is one way to move $a_5$ down (namely, to $a_6$), while $a_1$ and $a_2$ cannot be moved down. 

For the vertex $\{a_1,a_3, a_4\}$ in $\Conf{3}{\gG}$ there are two possible ways to move down: we can go from $a_4$ to $a_5$ and $a_3$ to $a_2$. As these 
 two movements are independent of each other (i.e.~the edges of $\G$ involved are disjoint) this corresponds to a descending square in 
  $\Conf{3}{\gG}$, as shown in the middle in Figure~\ref{fig:c6_egv1}, 
the descending link of $\{a_1, a_3, a_4\}$ is the red edge.

Finally, for the vertex $\{a_2, a_4, a_6\}$ in $\Conf{3}{\gG}$, there are four possible ways to move down:  $a_6$ and $a_2$ can move  to $a_1$, and  $a_4$ can move to $a_5$ or $a_3$. We cannot move $a_6$ at the same time as $a_2$, but can move $a_4$ independently of both. Thus, the descending cells at $\{a_2, a_4, a_6\}$ consist of four 2-cells, as shown in Figure~\ref{fig:c6_egv1}, and the descending link is four edges forming a circle. 
\end{example}

We end this section with the following observation:
\begin{obs}\label{o:cone}
Given a vertex $\{u, v, w\}$ of $\Conf{3}{\G}$, if one vertex, say $u$, 
has exactly one edge to move down along in $\G$, 
and this motion is independent of any moves involving $v$ and $w$, then the descending link  of $\{u, v, w\}$ is a cone, with the cone point coming from the edge incident to $u$, and the base coming from descending cells which only involve the movement of $v$ and $w$.  
\end{obs}
\section{Sun Graphs} \label{sec:sun}
In this section, we show that the 3-strand braid group of a sun graph is free and compute the Euler characteristic of the corresponding configuration space in a special case. 

\begin{Definition} \label{def:sun}
The sun graph $\mathcal{S}(x_1,\ldots, x_n)$ is a $2n$-cycle with vertices labeled in order by $v_1,\ldots,v_{2n}$, where $v_{2i-1}$ has $x_i$ additional edges called \textit{rays} attached to it. 
The endpoints of those rays will be denoted by  $v_{2i-1,k}$. 
\end{Definition}

\begin{figure}
\begin{tikzpicture}[scale = 0.8]
%%% edges
\draw[thick] (0,-1) -- (2.5,1) -- (2.5,3) -- (0,5) -- (-2.5,3) --(-2.5,1) -- (0,-1);
%% vertices
\draw[fill=black] (0,-1) circle (3pt) node[below=0.1] { $v_7$};
%\draw[fill=black] (1.25,0) circle (3pt);%node[below right=0.06] { $v_6$};
\draw[fill=black] (2.5,1) circle (3pt)node[left=0.1] { $v_5$};
%\draw[fill=black] (2.5,2) circle (3pt) node[left=0.05] { $v_4$};
\draw[fill=black] (2.5,3) circle (3pt) node[left=0.05] { $v_3$};
%\draw[fill=black] (1.25,4) circle (3pt) node[below left=0.06] { $v_2$};
\draw[fill=black] (0,5) circle (3pt) node[below=0.1] { $v_1$};
%%\draw[fill=black] (-1.25,4) circle (3pt) node[below right=0.06] { $v_{12}$};
\draw[fill=black] (-2.5,3) circle (3pt) node[right=0.1] { $v_{11}$};
%\draw[fill=black] (-2.5,2) circle (3pt) node[right=0.1] { $v_{10}$};
\draw[fill=black] (-2.5,1) circle (3pt) node[right=0.1] { $v_{9}$};
%\draw[fill=black] (-1.25,0) circle (3pt) node[below left=0.06] { $v_{8}$};

\draw[thick] (2.5,3) -- (2.8, 4);
\draw[fill=black] (2.8, 4) circle (3pt) ;%node[above=0.1] { $v_{3,1}$};

\draw (0,5) -- (-1,6);
\draw (0,5) -- (0,6);
\draw (0,5) -- (1,6);
\draw[fill=black] (-1,6) circle (3pt) ;%node[above=0.1] { $v_{1,1}$};
\draw[fill=black] (0,6) circle (3pt) ;%node[above=0.1] { $v_{1,2}$};
\draw[fill=black] (1,6) circle (3pt) ;%node[above=0.1] { $v_{1,3}$};

\draw[fill=black] (3.2,2) circle (3pt) ;%node[above=0.05] { $v_{5,1}$};
\draw[fill=black]  (3.9,2) circle (3pt) ;%node[above=0.05] { $v_{5,2}$};
\draw[fill=black]  (4.6,2) circle (3pt) ;%node[above=0.05] { $v_{5,3}$};
\draw (2.5,1) -- (3.2,2);
\draw (2.5,1) -- (3.9,2);
\draw (2.5,1) -- (4.6,2);

\draw[fill=black] (-.4, 0) circle (3pt) ;%node[above=0.05] { $v_{7,1}$};
\draw[fill=black] (0.4,0) circle (3pt) ;%node[above=0.05] { $v_{7,2}$};
\draw (0,-1) -- (-.4, 0);
\draw (0,-1) -- (0.4,0);

\draw[fill=black] (-3.3,4) circle (3pt) ;%node[above=0.1] { $v_{11,1}$};
\draw[fill=black] (-2.2,4) circle (3pt) ;%node[above=0.1] { $v_{11,2}$};
\draw (-2.5,3) -- (-2.2,4);
\draw (-2.5,3) -- (-3.3,4);
\end{tikzpicture}
    \caption{The sun graph $\mathcal{S}(3,1,3,2,0,2)$, where $n=6$. An admissible subdivision is obtained by subdividing each edge once (the even-labeled vertices are hence missing).
    For the inductive part of the proof of Theorem~\ref{t:sun}, we will attach a ray to the top vertex of the cycle.}     \label{SunGraph}
\end{figure}
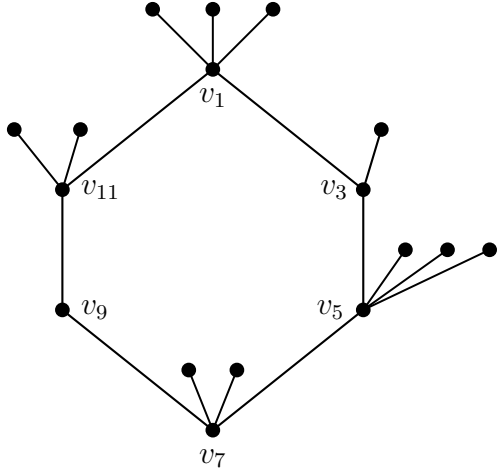

We now show that if $\mathcal{S}$ is a sun graph, then $B_3(\mathcal S)$ is free.  This proves Theorem~\ref{t:sun} of the introduction. We make the convention that if $v$ is a degree $2$ vertex of $\gG$, then its image under a Morse function $f: \gG \rightarrow \mathbb{R}$ will be the average of the image of its neighbors.

\begin{proof}[Proof of Theorem~\ref{t:sun}]
We will prove this via induction on the total number of rays, $\sum_{i=1}^{n}x_i$. 
In the base case, $\mathcal{S}(0,\ldots,0)$ is a cycle.   To get an admissible subdivision of a cycle, it is enough to cellulate it with 4 vertices, and we denote this by  $C_4$. 
Then $\text{Conf}_3^\square(C_4)$ is $C_4$ as well. Since $\text{Conf}_3^\square(C_4)$ is homotopy equivalent to $\text{Conf}_3(\mathcal{S}(0,\ldots,0))$, we have that $B_3(\mathcal{S}(0,\ldots,0))=\mathbb{Z}$.

For the inductive step, let $\Gamma=\mathcal{S}(x_1,\ldots,x_n)$ and $\Gamma'=\mathcal{S}(x_1+1,\ldots,x_n)$ and assume that  
$\Conf{3}{\Gamma}$ is homotopy equivalent to a wedge of circles (so
$B_3(\Gamma)$ is free). We will show that $B_3(\Gamma')$ is also free. Note that $\G'$ is obtained from $\G$ by adding a subdivided edge $e$ to $\Gamma$ as a new ray at $v_1$.  
We denote the degree $2$ vertex of this new edge $e$ by $v$ and the degree $1$ vertex by $v'$. 
 Let $h$ be the Morse function on $\Gamma$ determined by setting $h(v_i)=|i-n-1|$ and $h(v_{2i-1,k})=h(v_{2i-1})+1$. 
 See Figure~\ref{SunGraph} for an example of an embedding inducing $h$ when $n=6$.)  
 We extend $h$ to a Morse function on $\Gamma'$ by letting $h(v') = 5n+8$ (so $h(v) = 3n+4$ by our convention). 
Let $f$ be the Morse function on $\text{Conf}_3^\square(\Gamma')$ induced by $h$, as described in Section~\ref{sec:conf Morse}. 

Note that the maximum of $h$ on $\G$ is $n+1$, while $h(v')> h(v) = 3n+4$.   So given a vertex $\{u_1, u_2, u_3\}$ of $\Conf{3}{\Gamma'}$, we have $f(\{u_1, u_2, u_3\})> 3n+3$ if and only if $u_i \neq v$ or $v'$ for any $i$.  
Consequently, ${f}^{-1}([0,3n+3])$ is homotopy equivalent to $\text{Conf}_3(\Gamma)$, which by induction is homotopy equivalent to a wedge of circles.
Therefore, by Corollary~\ref{c:freeproduct}, to prove that $B_3(\Gamma')$ is free, it is enough to show that the descending links of vertices of $\Conf{3}{\Gamma'}$ of the form $\{v, u, w\}$ or $\{v', u,w\}$ are unions of contractible spaces.  

For $\{v', u,w\}$, if neither $u$ nor $w$ is $v$, there is a unique way for $v'$ to move down, independently of $u$ and $w$, so the descending link is a cone by Observation~\ref{o:cone}.  For $\{v', v, w\}$ the descending link  is 2 points if either $h(w)=n$, (so  $w=v_1$) or  $h(w)= n+1$
(so $w$ is adjacent to $v_1$  on a ray), 
and it is a cone otherwise (as in all other cases $v$ can move down independently of $w$ and $v'$).

Thus we can assume that the vertex  is of the form $\{v, u, w\}$ with neither $u$ nor $w$ equal to $v'$.  Assume $h(w) \le h(u)$.  If $0< h(w) <n$ then $w$ has a unique way of moving down, independently of $u$ and $v$, so the descending link is a cone.  If $h(w) = 0$, it is either two points (if $h(u) = n$ or $n+1$) or a cone.  If $h(w) = n$, then it is 
either two points (if $h(u) = n+1$) or a cone.
a suspension, as $w$ has two ways to move down independently of $u$ and $v$.  Finally, if $h(w) >n$, so that both $u$ and $w$ are on rays connected to $v_1$,  it is either two points (if $h(u)=h(w) = n+1$) and a cone otherwise.

%Thus we can assume that the vertex  is of the form $\{v, u, w\}$ with neither $u$ nor $w$ equal to $v'$.  
%First suppose $u$ is on a ray connected to $v_1$. 
%If $w$ is either $v_1$ or $v_{n+1}$ (the lowest vertex of $\G$) then the descending link is either 2 points or a cone on $2$ points.  Similarly,  if $w$ is also on a ray connected to $e$ the links is either 3 points or a cone. 
%For all other $w$, there is a unique way for $w$ to move down in $\G$, and this is independent of moving down $u$ and $v$, so the descending link is a cone 
%by Observation~\ref{o:cone}. 

Secondly, if $u = v_1$, then the descending link is either $2$ points (if $w$ is adjacent to $v_1$ or $w=v_{n+1}$ or a cone on 2 points (for all other $w$, by Observation~\ref{o:cone}).   
Finally, when neither $u$ nor $w$ is $v_1$ or on a ray attached to $v_1$, we see that $v$ can move down in exactly one way,  independently of $u$ and $w$, so the descending link is again a cone.

Therefore, in each case, the descending link is a union of contractible spaces, so $B_3(\Gamma')$ is free by 
Corollary~\ref{c:freeproduct}. 
\end{proof}

\section{Pulsar graphs}
\label{sec:pulsar}
In this section, we consider the 3-strand graph braid groups of pulsar graphs.  We compute the Euler characteristic of the corresponding configuration spaces in the special case of generalized theta graphs, and show explicitly that $\Conf{3}{\Theta_4}$ is homeomorphic to a closed, orientable surface of genus $3$ when $\Theta_4$ is given the simplest possible admissible subdivision.   Finally, we obtain a free-product decomposition for the 3-strand braid groups of general pulsar graphs.

\begin{Definition}\label{def:pulsar}
A \textit{pulsar graph}, denoted $\cP_{m, n_1, n_2}$, is obtained by gluing $m-1$ cycles along a fixed segment, and attaching $n_1$ and $n_2$ rays respectively to the two endpoints of this segment, see Figure~\ref{fig:theta_m} (right). In the  special case when $n_1=n_2=0$, we obtain a \emph{generalized theta graph}, denoted $\Theta_m$, see Figure~\ref{fig:theta_m} (left).  

An admissible subdivision on $\cP_{m, n_1, n_2}$ is obtained by adding a single degree 2 vertex to each edge, as shown in 
Figure~\ref{fig:theta_m}.  We  
 label the suspension vertices of $\Theta_m$ by $a_1$ and $a_2$ and the degree $2$ vertices of $\Theta_m$ by $b_1, \dots b_m$.  
 Finally, we label the vertices on the rays incident to $a_1$ and $a_2$ by $c_i, c_i'$ and $d_i, d_i'$ respectively, for $i$ in the appropriate range, as shown on the right in Figure~\ref{fig:theta_m}
  \end{Definition}

\begin{figure}%[htb]
\tikzset{every picture/.style={line width=0.75pt}} %set default line width to 0.75pt        

\begin{tikzpicture}[x=0.75pt,y=0.75pt,yscale=-0.8,xscale=0.8]
%uncomment if require: \path (0,438); %set diagram left start at 0, and has height of 438

%Shape: Ellipse [id:dp8492744277893345] 
\draw   (385.57,174.74) .. controls (385.57,125.67) and (415.26,85.88) .. (451.88,85.88) .. controls (488.51,85.88) and (518.2,125.67) .. (518.2,174.74) .. controls (518.2,223.82) and (488.51,263.6) .. (451.88,263.6) .. controls (415.26,263.6) and (385.57,223.82) .. (385.57,174.74) -- cycle ;
%Curve Lines [id:da8622060716785824] 
\draw    (451.88,85.88) .. controls (402.18,145.21) and (399.41,197.4) .. (451.88,263.6) ;
%Curve Lines [id:da343105473567032] 
\draw    (451.88,85.88) .. controls (445.2,154.6) and (444.2,177.6) .. (451.88,263.6) ;
%Shape: Circle [id:dp39557313088205137] 
\draw  [fill={rgb, 255:red, 0; green, 0; blue, 0 }  ,fill opacity=1 ] (447.53,85.88) .. controls (447.53,83.48) and (449.48,81.52) .. (451.88,81.52) .. controls (454.29,81.52) and (456.24,83.48) .. (456.24,85.88) .. controls (456.24,88.29) and (454.29,90.24) .. (451.88,90.24) .. controls (449.48,90.24) and (447.53,88.29) .. (447.53,85.88) -- cycle ;
%Shape: Circle [id:dp44199481779479766] 
\draw  [fill={rgb, 255:red, 0; green, 0; blue, 0 }  ,fill opacity=1 ] (447.53,263.6) .. controls (447.53,261.19) and (449.48,259.24) .. (451.88,259.24) .. controls (454.29,259.24) and (456.24,261.19) .. (456.24,263.6) .. controls (456.24,266.01) and (454.29,267.96) .. (451.88,267.96) .. controls (449.48,267.96) and (447.53,266.01) .. (447.53,263.6) -- cycle ;
%Shape: Circle [id:dp7447676261868432] 
\draw  [fill={rgb, 255:red, 0; green, 0; blue, 0 }  ,fill opacity=1 ] (381.21,174.74) .. controls (381.21,172.33) and (383.16,170.38) .. (385.57,170.38) .. controls (387.98,170.38) and (389.93,172.33) .. (389.93,174.74) .. controls (389.93,177.15) and (387.98,179.1) .. (385.57,179.1) .. controls (383.16,179.1) and (381.21,177.15) .. (381.21,174.74) -- cycle ;
%Shape: Circle [id:dp7081952666697673] 
\draw  [fill={rgb, 255:red, 0; green, 0; blue, 0 }  ,fill opacity=1 ] (408.16,174.59) .. controls (408.16,172.19) and (410.11,170.24) .. (412.52,170.24) .. controls (414.92,170.24) and (416.87,172.19) .. (416.87,174.59) .. controls (416.87,177) and (414.92,178.95) .. (412.52,178.95) .. controls (410.11,178.95) and (408.16,177) .. (408.16,174.59) -- cycle ;
%Shape: Circle [id:dp19046297443756566] 
\draw  [fill={rgb, 255:red, 0; green, 0; blue, 0 }  ,fill opacity=1 ] (443.17,174.74) .. controls (443.17,172.33) and (445.12,170.38) .. (447.53,170.38) .. controls (449.93,170.38) and (451.88,172.33) .. (451.88,174.74) .. controls (451.88,177.15) and (449.93,179.1) .. (447.53,179.1) .. controls (445.12,179.1) and (443.17,177.15) .. (443.17,174.74) -- cycle ;
%Shape: Circle [id:dp2216785835076478] 
\draw  [fill={rgb, 255:red, 0; green, 0; blue, 0 }  ,fill opacity=1 ] (513.84,174.74) .. controls (513.84,172.33) and (515.79,170.38) .. (518.2,170.38) .. controls (520.61,170.38) and (522.56,172.33) .. (522.56,174.74) .. controls (522.56,177.15) and (520.61,179.1) .. (518.2,179.1) .. controls (515.79,179.1) and (513.84,177.15) .. (513.84,174.74) -- cycle ;
%Shape: Circle [id:dp250439200726871] 
\draw  [fill={rgb, 255:red, 0; green, 0; blue, 0 }  ,fill opacity=1 ] (396.21,43.74) .. controls (396.21,41.33) and (398.16,39.38) .. (400.57,39.38) .. controls (402.98,39.38) and (404.93,41.33) .. (404.93,43.74) .. controls (404.93,46.15) and (402.98,48.1) .. (400.57,48.1) .. controls (398.16,48.1) and (396.21,46.15) .. (396.21,43.74) -- cycle ;
%Shape: Circle [id:dp6792163604717705] 
\draw  [fill={rgb, 255:red, 0; green, 0; blue, 0 }  ,fill opacity=1 ] (500.21,43.74) .. controls (500.21,41.33) and (502.16,39.38) .. (504.57,39.38) .. controls (506.98,39.38) and (508.93,41.33) .. (508.93,43.74) .. controls (508.93,46.15) and (506.98,48.1) .. (504.57,48.1) .. controls (502.16,48.1) and (500.21,46.15) .. (500.21,43.74) -- cycle ;
%Straight Lines [id:da12390467744272438] 
\draw    (400.57,43.74) -- (451.88,85.88) ;
\draw    (504.57,43.74) -- (451.88,85.88) ;

%Shape: Ellipse [id:dp2286366609823789] 
\draw   (129.57,174.75) .. controls (129.57,125.67) and (159.26,85.89) .. (195.88,85.89) .. controls (232.51,85.89) and (262.2,125.67) .. (262.2,174.75) .. controls (262.2,223.83) and (232.51,263.61) .. (195.88,263.61) .. controls (159.26,263.61) and (129.57,223.83) .. (129.57,174.75) -- cycle ;
%Curve Lines [id:da1594114160456921] 
\draw    (195.88,85.89) .. controls (146.18,145.22) and (143.41,197.41) .. (195.88,263.61) ;
%Curve Lines [id:da7828867259150625] 
\draw    (195.88,85.89) .. controls (189.2,154.61) and (188.2,177.61) .. (195.88,263.61) ;
%Shape: Circle [id:dp5546345027844011] 
\draw  [fill={rgb, 255:red, 0; green, 0; blue, 0 }  ,fill opacity=1 ] (191.53,85.89) .. controls (191.53,83.48) and (193.48,81.53) .. (195.88,81.53) .. controls (198.29,81.53) and (200.24,83.48) .. (200.24,85.89) .. controls (200.24,88.3) and (198.29,90.25) .. (195.88,90.25) .. controls (193.48,90.25) and (191.53,88.3) .. (191.53,85.89) -- cycle ;
%Shape: Circle [id:dp7379689186279879] 
\draw  [fill={rgb, 255:red, 0; green, 0; blue, 0 }  ,fill opacity=1 ] (191.53,263.61) .. controls (191.53,261.2) and (193.48,259.25) .. (195.88,259.25) .. controls (198.29,259.25) and (200.24,261.2) .. (200.24,263.61) .. controls (200.24,266.02) and (198.29,267.97) .. (195.88,267.97) .. controls (193.48,267.97) and (191.53,266.02) .. (191.53,263.61) -- cycle ;
%Shape: Circle [id:dp296118793473237] 
\draw  [fill={rgb, 255:red, 0; green, 0; blue, 0 }  ,fill opacity=1 ] (125.21,174.75) .. controls (125.21,172.34) and (127.16,170.39) .. (129.57,170.39) .. controls (131.98,170.39) and (133.93,172.34) .. (133.93,174.75) .. controls (133.93,177.16) and (131.98,179.11) .. (129.57,179.11) .. controls (127.16,179.11) and (125.21,177.16) .. (125.21,174.75) -- cycle ;
%Shape: Circle [id:dp8237215457391385] 
\draw  [fill={rgb, 255:red, 0; green, 0; blue, 0 }  ,fill opacity=1 ] (152.16,174.6) .. controls (152.16,172.2) and (154.11,170.24) .. (156.52,170.24) .. controls (158.92,170.24) and (160.87,172.2) .. (160.87,174.6) .. controls (160.87,177.01) and (158.92,178.96) .. (156.52,178.96) .. controls (154.11,178.96) and (152.16,177.01) .. (152.16,174.6) -- cycle ;
%Shape: Circle [id:dp676162028702737] 
\draw  [fill={rgb, 255:red, 0; green, 0; blue, 0 }  ,fill opacity=1 ] (187.17,174.75) .. controls (187.17,172.34) and (189.12,170.39) .. (191.53,170.39) .. controls (193.93,170.39) and (195.88,172.34) .. (195.88,174.75) .. controls (195.88,177.16) and (193.93,179.11) .. (191.53,179.11) .. controls (189.12,179.11) and (187.17,177.16) .. (187.17,174.75) -- cycle ;
%Shape: Circle [id:dp9437218375103451] 
\draw  [fill={rgb, 255:red, 0; green, 0; blue, 0 }  ,fill opacity=1 ] (257.84,174.75) .. controls (257.84,172.34) and (259.79,170.39) .. (262.2,170.39) .. controls (264.61,170.39) and (266.56,172.34) .. (266.56,174.75) .. controls (266.56,177.16) and (264.61,179.11) .. (262.2,179.11) .. controls (259.79,179.11) and (257.84,177.16) .. (257.84,174.75) -- cycle ;

% Text Node
\draw (212.24,172.4) node [anchor=north west][inner sep=0.75pt]    {$\cdots $};
% Text Node
\draw (197.88,98.65) node [anchor=north west][inner sep=0.75pt]    {$a_{1}$};
% Text Node
\draw (198.53,245.01) node [anchor=north west][inner sep=0.75pt]    {$a_{2}$};
% Text Node
\draw (105,172.41) node [anchor=north west][inner sep=0.75pt]    {$b_{1}$};
% Text Node
\draw (137,172.41) node [anchor=north west][inner sep=0.75pt]    {$b_{2}$};
% Text Node
\draw (169,172.41) node [anchor=north west][inner sep=0.75pt]    {$b_{3}$};
% Text Node
\draw (273,172.41) node [anchor=north west][inner sep=0.75pt]    {$b_{m}$};
% Text Node
\draw (470.24,172.4) node [anchor=north west][inner sep=0.75pt]    {$\cdots $};
% Text Node
\draw (453.88,98.64) node [anchor=north west][inner sep=0.75pt]    {$a_{1}$};
% Text Node
\draw (454.53,245) node [anchor=north west][inner sep=0.75pt]    {$a_{2}$};
% Text Node
\draw (361,172.4) node [anchor=north west][inner sep=0.75pt]    {$b_{1}$};
% Text Node
\draw (393,172.4) node [anchor=north west][inner sep=0.75pt]    {$b_{2}$};
% Text Node
\draw (425,172.4) node [anchor=north west][inner sep=0.75pt]    {$b_{3}$};
% Text Node
\draw (525,172.4) node [anchor=north west][inner sep=0.5pt]    {$b_{m}$};

\draw (560,172.4) node [anchor=north west][inner sep=0.75pt]    {$d'_1$};

\draw (620,172.4) node [anchor=north west][inner sep=0.75pt]    {$d'_{n_2}$};
\draw (454,265) .. controls (525,260) and (560, 190) .. (560,172.4);
\draw (454,265) .. controls (575,260) and (620, 210) .. (620,172.4);

\fill[black] (560,174) circle (4.85);
\draw (580, 174) node [anchor=north west][inner sep=0.75pt]    {$\dotsc $};
\fill[black] (620,174) circle (4.85);

%New vertices and labels
\fill[black] (535,222) circle (4.85);
\draw (535.5,222.5) node [anchor=north west][inner sep=0.75pt]    {$d_1$};

\fill[black] (594,222) circle (4.85);
\draw (594.5,222.5) node [anchor=north west][inner sep=0.75pt]    {$d_{n_2}$};
\fill[black] (425, 65) circle (4.85);
\draw (425.6, 65.4) node [anchor= north east][inner sep=0.75pt]    {$c_1$};
\fill[black] (478,65) circle (4.85);
\draw (478.6 ,65.4) node [anchor=north west][inner sep=0.75pt]    {$c_{n_1}$};

% Text Node
\draw (425,39.4) node [anchor=north west][inner sep=0.75pt]    {$\dotsc \dotsc $};
% Text Node
\draw (374,34) node [anchor=north west][inner sep=0.75pt]    {$c'_{1}$};
% Text Node
\draw (510, 34) node [anchor=north west][inner sep=0.75pt]    {$c'_{n_{1}}$};

\end{tikzpicture}
    \caption{The generalized theta graph $\Theta_m$ and the pulsar graph $\cP_{m, n_1, n_2}$.}
    \label{fig:theta_m}
    \end{figure}
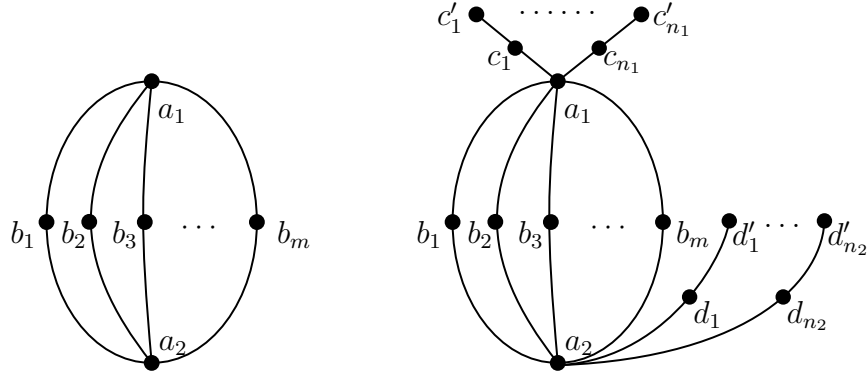

\subsection{Euler characteristics}
We now calculate the  %number of vertices, edges, and faces in 
Euler characteristic of $\cConf_3(\Theta_m)$.

Since the number of vertices in $\Theta_m$ is $m+2$, the number of vertices in $\Conf{3}{\Theta_m}$  is 
$\binom{m+2}{3}$.  Similarly, there are $2m$ edges in $\Theta_m$, so the number of edges in $\Conf{3}{\Theta_m}$ is 
$(2m)\binom{m}{2}$.  Finally, the number of pairs of disjoint edges in $\Theta_m$ is 
$ m^2 - m$, so the number of  2-cells in $\Conf{3}{\Theta_m}$  is 
$( m^2 - m)(m-2)$.  As $\Theta_m$ has no triples of disjoint edges, $\Conf{3}{\Theta_m}$ has no 3-cells. 

Therefore we have (after simplifying):
\begin{equation}\label{eq:chitheta}
\chi(\cConf_3(\Theta_m))=
\frac{m(m-2)(m-7)}{6}.
\end{equation}

In particular, $\chi(\cConf_3({\Theta_7})) =0$ and $\chi(\cConf_3{\Theta_m}) >0$ when $m>7$.

\subsection{The case of $\Theta_4$}

We now show that $\cConf^\square_3(\Theta_4)$ is a closed surface by showing a neighborhood of each vertex is a closed 2-dimensional disk. 
There are three types of vertices, namely $\{b_i,b_j,b_k\}$, $\{a_i,b_j,b_k\}$, and $\{a_1,a_2,b_i\}$ (in the notation of Figure~\ref{fig:theta_m}) and their neighborhoods are shown in Figures~\ref{fig:case1} and~\ref{fig:case3}.

    \begin{figure}%[htb]
\tikzset{every picture/.style={line width=0.75pt}} %set default line width to 0.75pt        

\begin{tikzpicture}[x=0.75pt,y=0.75pt,yscale=-1,xscale=1]
%uncomment if require: \path (0,584); %set diagram left start at 0, and has height of 584

%Shape: Star [id:dp6110868587570784] 
\draw  [color={rgb, 255:red, 0; green, 0; blue, 0 }  ,draw opacity=1 ] (189.6,70.12) -- (221.5,136.16) -- (300.1,128.36) -- (253.4,186.59) -- (300.1,244.82) -- (221.5,237.02) -- (189.6,303.06) -- (157.7,237.02) -- (79.1,244.82) -- (125.8,186.59) -- (79.1,128.36) -- (157.7,136.16) -- cycle ;
%Straight Lines [id:da21304575948204452] 
\draw [color={rgb, 255:red, 0; green, 0; blue, 0 }  ,draw opacity=1 ]   (157.7,136.16) -- (221.5,237.02) ;
%Straight Lines [id:da9391749327033612] 
\draw [color={rgb, 255:red, 0; green, 0; blue, 0 }  ,draw opacity=1 ]   (221.5,136.16) -- (157.7,237.02) ;
%Straight Lines [id:da7648791098773933] 
\draw [color={rgb, 255:red, 0; green, 0; blue, 0 }  ,draw opacity=1 ]   (125.8,186.59) -- (253.4,186.59) ;
%Shape: Ellipse [id:dp6554936861733469] 
\draw  [color={rgb, 255:red, 208; green, 2; blue, 27 }  ,draw opacity=1 ][fill={rgb, 255:red, 208; green, 2; blue, 27 }  ,fill opacity=1 ] (186.31,186.59) .. controls (186.31,184.77) and (187.78,183.3) .. (189.6,183.3) .. controls (191.42,183.3) and (192.89,184.77) .. (192.89,186.59) .. controls (192.89,188.41) and (191.42,189.88) .. (189.6,189.88) .. controls (187.78,189.88) and (186.31,188.41) .. (186.31,186.59) -- cycle ;
%Shape: Ellipse [id:dp9673759432711233] 
\draw  [color={rgb, 255:red, 0; green, 0; blue, 0 }  ,draw opacity=1 ][fill={rgb, 255:red, 0; green, 0; blue, 0 }  ,fill opacity=1 ] (186.31,70.12) .. controls (186.31,68.3) and (187.78,66.83) .. (189.6,66.83) .. controls (191.42,66.83) and (192.89,68.3) .. (192.89,70.12) .. controls (192.89,71.94) and (191.42,73.41) .. (189.6,73.41) .. controls (187.78,73.41) and (186.31,71.94) .. (186.31,70.12) -- cycle ;
%Shape: Ellipse [id:dp7224936376422009] 
\draw  [color={rgb, 255:red, 0; green, 0; blue, 0 }  ,draw opacity=1 ][fill={rgb, 255:red, 0; green, 0; blue, 0 }  ,fill opacity=1 ] (218.21,136.16) .. controls (218.21,134.34) and (219.68,132.87) .. (221.5,132.87) .. controls (223.32,132.87) and (224.79,134.34) .. (224.79,136.16) .. controls (224.79,137.98) and (223.32,139.45) .. (221.5,139.45) .. controls (219.68,139.45) and (218.21,137.98) .. (218.21,136.16) -- cycle ;
%Shape: Ellipse [id:dp4620498684678651] 
\draw  [color={rgb, 255:red, 0; green, 0; blue, 0 }  ,draw opacity=1 ][fill={rgb, 255:red, 0; green, 0; blue, 0 }  ,fill opacity=1 ] (296.81,128.36) .. controls (296.81,126.54) and (298.29,125.06) .. (300.1,125.06) .. controls (301.92,125.06) and (303.4,126.54) .. (303.4,128.36) .. controls (303.4,130.17) and (301.92,131.65) .. (300.1,131.65) .. controls (298.29,131.65) and (296.81,130.17) .. (296.81,128.36) -- cycle ;
%Shape: Ellipse [id:dp7402260842548145] 
\draw  [color={rgb, 255:red, 0; green, 0; blue, 0 }  ,draw opacity=1 ][fill={rgb, 255:red, 0; green, 0; blue, 0 }  ,fill opacity=1 ] (250.11,186.59) .. controls (250.11,184.77) and (251.58,183.3) .. (253.4,183.3) .. controls (255.22,183.3) and (256.69,184.77) .. (256.69,186.59) .. controls (256.69,188.41) and (255.22,189.88) .. (253.4,189.88) .. controls (251.58,189.88) and (250.11,188.41) .. (250.11,186.59) -- cycle ;
%Shape: Ellipse [id:dp7074560514522676] 
\draw  [color={rgb, 255:red, 0; green, 0; blue, 0 }  ,draw opacity=1 ][fill={rgb, 255:red, 0; green, 0; blue, 0 }  ,fill opacity=1 ] (296.81,244.82) .. controls (296.81,243.01) and (298.29,241.53) .. (300.1,241.53) .. controls (301.92,241.53) and (303.4,243.01) .. (303.4,244.82) .. controls (303.4,246.64) and (301.92,248.12) .. (300.1,248.12) .. controls (298.29,248.12) and (296.81,246.64) .. (296.81,244.82) -- cycle ;
%Shape: Ellipse [id:dp06846798671853227] 
\draw  [color={rgb, 255:red, 0; green, 0; blue, 0 }  ,draw opacity=1 ][fill={rgb, 255:red, 0; green, 0; blue, 0 }  ,fill opacity=1 ] (218.21,237.02) .. controls (218.21,235.21) and (219.68,233.73) .. (221.5,233.73) .. controls (223.32,233.73) and (224.79,235.21) .. (224.79,237.02) .. controls (224.79,238.84) and (223.32,240.31) .. (221.5,240.31) .. controls (219.68,240.31) and (218.21,238.84) .. (218.21,237.02) -- cycle ;
%Shape: Circle [id:dp2741204363707883] 
\draw  [color={rgb, 255:red, 0; green, 0; blue, 0 }  ,draw opacity=1 ][fill={rgb, 255:red, 0; green, 0; blue, 0 }  ,fill opacity=1 ] (186.31,303.06) .. controls (186.31,301.24) and (187.78,299.77) .. (189.6,299.77) .. controls (191.42,299.77) and (192.89,301.24) .. (192.89,303.06) .. controls (192.89,304.88) and (191.42,306.35) .. (189.6,306.35) .. controls (187.78,306.35) and (186.31,304.88) .. (186.31,303.06) -- cycle ;
%Shape: Ellipse [id:dp7017366756026033] 
\draw  [color={rgb, 255:red, 0; green, 0; blue, 0 }  ,draw opacity=1 ][fill={rgb, 255:red, 0; green, 0; blue, 0 }  ,fill opacity=1 ] (154.41,237.02) .. controls (154.41,235.21) and (155.88,233.73) .. (157.7,233.73) .. controls (159.52,233.73) and (160.99,235.21) .. (160.99,237.02) .. controls (160.99,238.84) and (159.52,240.31) .. (157.7,240.31) .. controls (155.88,240.31) and (154.41,238.84) .. (154.41,237.02) -- cycle ;
%Shape: Ellipse [id:dp31123432244332294] 
\draw  [color={rgb, 255:red, 0; green, 0; blue, 0 }  ,draw opacity=1 ][fill={rgb, 255:red, 0; green, 0; blue, 0 }  ,fill opacity=1 ] (154.41,136.16) .. controls (154.41,134.34) and (155.88,132.87) .. (157.7,132.87) .. controls (159.52,132.87) and (160.99,134.34) .. (160.99,136.16) .. controls (160.99,137.98) and (159.52,139.45) .. (157.7,139.45) .. controls (155.88,139.45) and (154.41,137.98) .. (154.41,136.16) -- cycle ;
%Shape: Ellipse [id:dp23279533142346853] 
\draw  [color={rgb, 255:red, 0; green, 0; blue, 0 }  ,draw opacity=1 ][fill={rgb, 255:red, 0; green, 0; blue, 0 }  ,fill opacity=1 ] (75.8,128.36) .. controls (75.8,126.54) and (77.28,125.06) .. (79.1,125.06) .. controls (80.91,125.06) and (82.39,126.54) .. (82.39,128.36) .. controls (82.39,130.17) and (80.91,131.65) .. (79.1,131.65) .. controls (77.28,131.65) and (75.8,130.17) .. (75.8,128.36) -- cycle ;
%Shape: Ellipse [id:dp1426090377922853] 
\draw  [color={rgb, 255:red, 0; green, 0; blue, 0 }  ,draw opacity=1 ][fill={rgb, 255:red, 0; green, 0; blue, 0 }  ,fill opacity=1 ] (122.51,186.59) .. controls (122.51,184.77) and (123.98,183.3) .. (125.8,183.3) .. controls (127.62,183.3) and (129.09,184.77) .. (129.09,186.59) .. controls (129.09,188.41) and (127.62,189.88) .. (125.8,189.88) .. controls (123.98,189.88) and (122.51,188.41) .. (122.51,186.59) -- cycle ;
%Shape: Ellipse [id:dp8735024633491204] 
\draw  [color={rgb, 255:red, 0; green, 0; blue, 0 }  ,draw opacity=1 ][fill={rgb, 255:red, 0; green, 0; blue, 0 }  ,fill opacity=1 ] (75.8,244.82) .. controls (75.8,243.01) and (77.28,241.53) .. (79.1,241.53) .. controls (80.91,241.53) and (82.39,243.01) .. (82.39,244.82) .. controls (82.39,246.64) and (80.91,248.12) .. (79.1,248.12) .. controls (77.28,248.12) and (75.8,246.64) .. (75.8,244.82) -- cycle ;

%Shape: Star [id:dp39104990710494825] 
\draw  [color={rgb, 255:red, 0; green, 0; blue, 0 }  ,draw opacity=1 ] (485.6,66.12) -- (517.5,132.16) -- (596.1,124.36) -- (549.4,182.59) -- (596.1,240.82) -- (517.5,233.02) -- (485.6,299.06) -- (453.7,233.02) -- (375.1,240.82) -- (421.8,182.59) -- (375.1,124.36) -- (453.7,132.16) -- cycle ;
%Straight Lines [id:da9214771723664168] 
\draw [color={rgb, 255:red, 0; green, 0; blue, 0 }  ,draw opacity=1 ]   (453.7,132.16) -- (517.5,233.02) ;
%Straight Lines [id:da9660091962179143] 
\draw [color={rgb, 255:red, 0; green, 0; blue, 0 }  ,draw opacity=1 ]   (517.5,132.16) -- (453.7,233.02) ;
%Straight Lines [id:da8347794754037672] 
\draw [color={rgb, 255:red, 0; green, 0; blue, 0 }  ,draw opacity=1 ]   (421.8,182.59) -- (549.4,182.59) ;
%Shape: Ellipse [id:dp6885247913801538] 
\draw  [color={rgb, 255:red, 208; green, 2; blue, 27 }  ,draw opacity=1 ][fill={rgb, 255:red, 208; green, 2; blue, 27 }  ,fill opacity=1 ] (482.31,182.59) .. controls (482.31,180.77) and (483.78,179.3) .. (485.6,179.3) .. controls (487.42,179.3) and (488.89,180.77) .. (488.89,182.59) .. controls (488.89,184.41) and (487.42,185.88) .. (485.6,185.88) .. controls (483.78,185.88) and (482.31,184.41) .. (482.31,182.59) -- cycle ;
%Shape: Ellipse [id:dp09889072024070766] 
\draw  [color={rgb, 255:red, 0; green, 0; blue, 0 }  ,draw opacity=1 ][fill={rgb, 255:red, 0; green, 0; blue, 0 }  ,fill opacity=1 ] (482.31,66.12) .. controls (482.31,64.3) and (483.78,62.83) .. (485.6,62.83) .. controls (487.42,62.83) and (488.89,64.3) .. (488.89,66.12) .. controls (488.89,67.94) and (487.42,69.41) .. (485.6,69.41) .. controls (483.78,69.41) and (482.31,67.94) .. (482.31,66.12) -- cycle ;
%Shape: Ellipse [id:dp8134455053048955] 
\draw  [color={rgb, 255:red, 0; green, 0; blue, 0 }  ,draw opacity=1 ][fill={rgb, 255:red, 0; green, 0; blue, 0 }  ,fill opacity=1 ] (514.21,132.16) .. controls (514.21,130.34) and (515.68,128.87) .. (517.5,128.87) .. controls (519.32,128.87) and (520.79,130.34) .. (520.79,132.16) .. controls (520.79,133.98) and (519.32,135.45) .. (517.5,135.45) .. controls (515.68,135.45) and (514.21,133.98) .. (514.21,132.16) -- cycle ;
%Shape: Ellipse [id:dp2906667604190114] 
\draw  [color={rgb, 255:red, 0; green, 0; blue, 0 }  ,draw opacity=1 ][fill={rgb, 255:red, 0; green, 0; blue, 0 }  ,fill opacity=1 ] (592.81,124.36) .. controls (592.81,122.54) and (594.29,121.06) .. (596.1,121.06) .. controls (597.92,121.06) and (599.4,122.54) .. (599.4,124.36) .. controls (599.4,126.17) and (597.92,127.65) .. (596.1,127.65) .. controls (594.29,127.65) and (592.81,126.17) .. (592.81,124.36) -- cycle ;
%Shape: Ellipse [id:dp2933228990362793] 
\draw  [color={rgb, 255:red, 0; green, 0; blue, 0 }  ,draw opacity=1 ][fill={rgb, 255:red, 0; green, 0; blue, 0 }  ,fill opacity=1 ] (546.11,182.59) .. controls (546.11,180.77) and (547.58,179.3) .. (549.4,179.3) .. controls (551.22,179.3) and (552.69,180.77) .. (552.69,182.59) .. controls (552.69,184.41) and (551.22,185.88) .. (549.4,185.88) .. controls (547.58,185.88) and (546.11,184.41) .. (546.11,182.59) -- cycle ;
%Shape: Ellipse [id:dp02741289440494654] 
\draw  [color={rgb, 255:red, 0; green, 0; blue, 0 }  ,draw opacity=1 ][fill={rgb, 255:red, 0; green, 0; blue, 0 }  ,fill opacity=1 ] (592.81,240.82) .. controls (592.81,239.01) and (594.29,237.53) .. (596.1,237.53) .. controls (597.92,237.53) and (599.4,239.01) .. (599.4,240.82) .. controls (599.4,242.64) and (597.92,244.12) .. (596.1,244.12) .. controls (594.29,244.12) and (592.81,242.64) .. (592.81,240.82) -- cycle ;
%Shape: Ellipse [id:dp3980989651650755] 
\draw  [color={rgb, 255:red, 0; green, 0; blue, 0 }  ,draw opacity=1 ][fill={rgb, 255:red, 0; green, 0; blue, 0 }  ,fill opacity=1 ] (514.21,233.02) .. controls (514.21,231.21) and (515.68,229.73) .. (517.5,229.73) .. controls (519.32,229.73) and (520.79,231.21) .. (520.79,233.02) .. controls (520.79,234.84) and (519.32,236.31) .. (517.5,236.31) .. controls (515.68,236.31) and (514.21,234.84) .. (514.21,233.02) -- cycle ;
%Shape: Circle [id:dp11819904888241539] 
\draw  [color={rgb, 255:red, 0; green, 0; blue, 0 }  ,draw opacity=1 ][fill={rgb, 255:red, 0; green, 0; blue, 0 }  ,fill opacity=1 ] (482.31,299.06) .. controls (482.31,297.24) and (483.78,295.77) .. (485.6,295.77) .. controls (487.42,295.77) and (488.89,297.24) .. (488.89,299.06) .. controls (488.89,300.88) and (487.42,302.35) .. (485.6,302.35) .. controls (483.78,302.35) and (482.31,300.88) .. (482.31,299.06) -- cycle ;
%Shape: Ellipse [id:dp4198126871260637] 
\draw  [color={rgb, 255:red, 0; green, 0; blue, 0 }  ,draw opacity=1 ][fill={rgb, 255:red, 0; green, 0; blue, 0 }  ,fill opacity=1 ] (450.41,233.02) .. controls (450.41,231.21) and (451.88,229.73) .. (453.7,229.73) .. controls (455.52,229.73) and (456.99,231.21) .. (456.99,233.02) .. controls (456.99,234.84) and (455.52,236.31) .. (453.7,236.31) .. controls (451.88,236.31) and (450.41,234.84) .. (450.41,233.02) -- cycle ;
%Shape: Ellipse [id:dp384365167068458] 
\draw  [color={rgb, 255:red, 0; green, 0; blue, 0 }  ,draw opacity=1 ][fill={rgb, 255:red, 0; green, 0; blue, 0 }  ,fill opacity=1 ] (450.41,132.16) .. controls (450.41,130.34) and (451.88,128.87) .. (453.7,128.87) .. controls (455.52,128.87) and (456.99,130.34) .. (456.99,132.16) .. controls (456.99,133.98) and (455.52,135.45) .. (453.7,135.45) .. controls (451.88,135.45) and (450.41,133.98) .. (450.41,132.16) -- cycle ;
%Shape: Ellipse [id:dp7552358985673435] 
\draw  [color={rgb, 255:red, 0; green, 0; blue, 0 }  ,draw opacity=1 ][fill={rgb, 255:red, 0; green, 0; blue, 0 }  ,fill opacity=1 ] (371.8,124.36) .. controls (371.8,122.54) and (373.28,121.06) .. (375.1,121.06) .. controls (376.91,121.06) and (378.39,122.54) .. (378.39,124.36) .. controls (378.39,126.17) and (376.91,127.65) .. (375.1,127.65) .. controls (373.28,127.65) and (371.8,126.17) .. (371.8,124.36) -- cycle ;
%Shape: Ellipse [id:dp625666381371268] 
\draw  [color={rgb, 255:red, 0; green, 0; blue, 0 }  ,draw opacity=1 ][fill={rgb, 255:red, 0; green, 0; blue, 0 }  ,fill opacity=1 ] (418.51,182.59) .. controls (418.51,180.77) and (419.98,179.3) .. (421.8,179.3) .. controls (423.62,179.3) and (425.09,180.77) .. (425.09,182.59) .. controls (425.09,184.41) and (423.62,185.88) .. (421.8,185.88) .. controls (419.98,185.88) and (418.51,184.41) .. (418.51,182.59) -- cycle ;
%Shape: Ellipse [id:dp7348834114926523] 
\draw  [color={rgb, 255:red, 0; green, 0; blue, 0 }  ,draw opacity=1 ][fill={rgb, 255:red, 0; green, 0; blue, 0 }  ,fill opacity=1 ] (371.8,240.82) .. controls (371.8,239.01) and (373.28,237.53) .. (375.1,237.53) .. controls (376.91,237.53) and (378.39,239.01) .. (378.39,240.82) .. controls (378.39,242.64) and (376.91,244.12) .. (375.1,244.12) .. controls (373.28,244.12) and (371.8,242.64) .. (371.8,240.82) -- cycle ;

% Text Node
\draw (498.02,166.22) node [anchor=north west][inner sep=0.75pt]  [color={rgb, 255:red, 208; green, 2; blue, 27 }  ,opacity=1 ]  {$a_{1} a_{2} b_{i}$};
% Text Node
\draw (560.38,176.72) node [anchor=north west][inner sep=0.75pt]    {$b_{k} a_{2} b_{i}$};
% Text Node
\draw (370.76,178.72) node [anchor=north west][inner sep=0.75pt]    {$a_{1} b_{k} b_{i}$};
% Text Node
\draw (411.78,112.36) node [anchor=north west][inner sep=0.75pt]    {$b_{j} a_{2} b_{i}$};
% Text Node
\draw (519.5,244.71) node [anchor=north west][inner sep=0.75pt]    {$a_{1} b_{j} b_{i}$};
% Text Node
\draw (515.56,111.36) node [anchor=north west][inner sep=0.75pt]    {$a_{1} b_{l} b_{i}$};
% Text Node
\draw (416.28,245.82) node [anchor=north west][inner sep=0.75pt]    {$b_{l} a_{2} b_{i}$};
% Text Node
\draw (202.02,170.22) node [anchor=north west][inner sep=0.75pt]  [color={rgb, 255:red, 208; green, 2; blue, 27 }  ,opacity=1 ]  {$b_{i} b_{j} b_{k}$};
% Text Node
\draw (264.38,180.72) node [anchor=north west][inner sep=0.75pt]    {$a_{1} b_{j} b_{k}$};
% Text Node
\draw (74.76,182.72) node [anchor=north west][inner sep=0.75pt]    {$a_{2} b_{j} b_{k}$};
% Text Node
\draw (115.78,116.36) node [anchor=north west][inner sep=0.75pt]    {$b_{i} a_{1} b_{k}$};
% Text Node
\draw (223.5,248.71) node [anchor=north west][inner sep=0.75pt]    {$b_{i} a_{2} b_{k}$};
% Text Node
\draw (219.56,115.36) node [anchor=north west][inner sep=0.75pt]    {$b_{i} b_{j} a_{2}$};
% Text Node
\draw (120.28,249.82) node [anchor=north west][inner sep=0.75pt]    {$b_{i} b_{j} a_{1}$};

\end{tikzpicture}
    \caption{Local neighborhood of $\{b_i,b_j,b_k\}$ and $\{a_ib_jb_k\}$ in $\cConf^\square_3(\Theta_4)$.}
    \label{fig:case1}
    \end{figure}
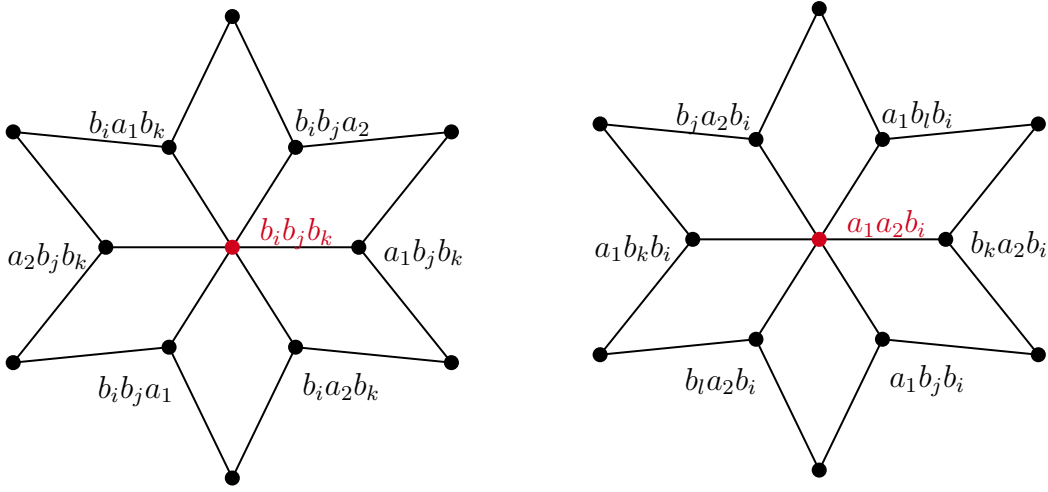

    \begin{figure}%[htb]
 \begin{tikzpicture}[xscale=0.6, yscale=0.6]

\fill[black] (0,3) circle (.13);
\fill[black] (3,0) circle (.13);
\fill[black] (3,3) circle (.13);
\fill[black] (0,-3) circle (.13);
\fill[black] (-3,0) circle (.13);
\fill[black] (-3,-3) circle (.13);
\fill[black] (3,-3) circle (.13);
\fill[black] (-3,3) circle (.13);
\draw[thick] (-3,3) -- (3, 3);
\draw[thick] (-3,3) -- (-3, -3);
\draw[thick] (0,3) -- (0, -3);
\draw[thick] (3,0) -- (-3, 0);
\draw[thick] (3,3) -- (3, -3);
\draw[thick] (-3,-3) -- (3, -3);
\fill[red] (0,0) circle (.13);
\node at (0.95,-.45) {$a_1b_ib_j$};
\node at (4.2,0) {$a_1a_2b_j$};
\node at (-4.1,0) {$a_1a_2b_i$};
\node at (0, 3.6) {$b_ib_jb_k$};
\node at (0, -3.6) {$b_ib_jb_l$};
\end{tikzpicture}

    \caption{Local neighborhood of $\{a_i, b_j, b_k\}$ in $\cConf^\square_3(\Theta_4)$.}
    \label{fig:case3}
    \end{figure}
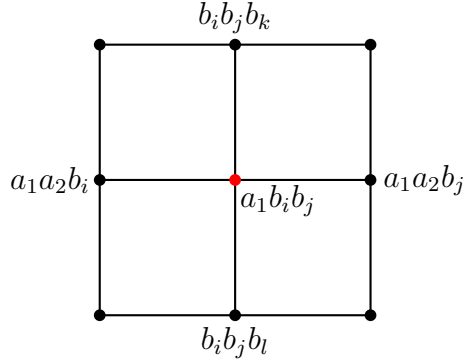

By the formula for Euler characteristic in~\eqref{eq:chitheta}, 
we have $\chi(\cConf^\square_3(\Theta_4))=-4$.  
As the Euler characteristic is even, the surface is orientable.  Hence $\cConf^\square_3(\Theta_4)$ is a closed, orientable surface of genus $3$.

\subsection{Free product decomposition}
In this section we show that if $\mathcal{P}_{m,n_1,n_2}$ is the pulsar graph from Definition~\ref{def:pulsar}, then the group $B_3(\mathcal{P}_{m,n_1,n_2})$ decomposes as a free product of the form $ B_3(\Theta_m)*\mathbb {F}$, where $\mathbb {F}$ is a free group. 
This proves Thoerem~\ref{t:split} of the introduction. 

\begin{proof}
   Let $m$ be fixed and consider the admissible subdivision on $\cP_{m,n_1,n_2}$ from Definition~\ref{def:pulsar}.  (See Figure~\ref{fig:theta_m}.)  
        Let us first assume that $n_2 = 0$. 
        
        Let    $h: \mathcal{P}_{m,n_1,0} \rightarrow \mathbb{R}$ be the function defined by 
          \[h(v)=\begin{cases}
        0, & v=a_2\\
        1,& v=b_j  \text{ for }1\leq  j \leq m\\
        2,& v=a_1\\
        5 &v=c_{i} \text{ for }1\leq  i \leq n_1\\
        6 &v=c'_{i} \text{ for }1\leq i\leq n_1
    \end{cases}.\] 
Let $f$ be the induced Morse function on $\Conf{3}{\mathcal{P}_{m,n_1,0}}$.  Observe that for a vertex $\{u, v, w\}$ of $\Conf{3}{\mathcal{P}_{m,n_1,0}}$, we have $f(\{u, v, w\}) > 4$ if and only if at least one of $u, v,$ and $w$ is equal to $c_i$ or $c_i'$ for some $i$.   In particular, $f^{-1}([0,4])$ is equal to
   $\Conf{3}{\Theta_m}$, which is connected.   With a view to applying Corollary~\ref{c:freeproduct}, we analyze the descending links of vertices in $\Conf{3}{\mathcal{P}_{m,n_1,0}}$ involving at least one vertex of the form $c_{i}$ or $c_i'$.

  First consider a vertex of the form $\{c'_{i}, v,w\}$.  If neither $v$ nor $w$ is $c_i$, then the descending link is a cone by Observation~\ref{o:cone}.  If $v = c_i$, then the descending link is 2 points if either $w = a_1$ or $w= c_j$ for some $j \neq i$, and is an edge for all other $w$.  
  
  Next, we consider vertices of the form $\{c_{i}, v, w\}$ where neither $v$ nor $w$ is equal $c_j'$ for any $j$, and suppose $h(w) \le h(v)$.  If $h(v) \le 1$, then $c_i$ has a unique way to move down, and the descending link is a cone.  Suppose 
  $h(v)= 2$ (equivalently $v= a_1$).  Then $w$ has a unique way to move down unless $h(w)=0$ (equivalently $w = a_2$), so the descending link is either a cone or (if $w = a_2$) a union of $m$-points.  
   If the vertex is of the form $\{c_{i}, c_{j}, w\}$, where $h(w) \le 2$ 
   then its descending link is either $m$ points (if $w = a_1)$ or the suspension of $0$ or $1$ points. 
   Finally, a vertex of the form $ \{c_{i}, c_{j}, c_{l}\}$ has descending link equal to $3$ points. 

From the above discussion, if at least one of $u, v,$ and $w$ is equal to $c_i$ or $c_i'$ for some $i$, then the descending link of 
$\{u, v, w\}$ is a union of contractible spaces.  Then Corollary~\ref{c:freeproduct} implies that 
  that $\Conf{3}{\mathcal{P}_{m,n_1,0}}$ is homotopy equivalent to a wedge of $\Conf{3}{\Theta_m}$ with a collection of circles, and the result follows (in the case $n_2=0$).  
  
  Now we assume that $n_2 > 0$, and 
   extend the Morse function $h$ above to $h:\mathcal{P}_{m,n_1,n_2}\to \mathbb{R}$ by setting $h(d_{i})= 0.5$ and $h(d_i')=1$  for all $i$ (see Figure \ref{fig:theta_m}).
Again, let $f$ denote the induced Morse function on $\Conf{3}{\mathcal{P}_{m,n_1,n_2}}$. 
Note that $f^{-1}([0,4])$ is homotopy equivalent to $\Conf{3}{\mathcal{P}_{m,0 ,n_2}}$, which, by the above argument applied to $\mathcal{P}_{m,0 ,n_2}$, is homotopy equivalent to a wedge of $\Conf{3}{\Theta_m}$ with a collection of circles.
Thus again we only need 
    to understand the descending links of vertices in $\Conf{3}{\mathcal{P}_{m,n_1,n_2}}$ involving at least one vertex of the form $c_{i}$ or $c_i'$. 
The analysis above applies verbatim to this case and shows that the descending links are all unions of contractible spaces. 
\end{proof}

\bibliographystyle{amsalpha}
\bibliography{refs}

\end{document}